\documentclass[12pt, english, a4paper]{amsart}

\usepackage{graphicx}

\usepackage{amsmath,amssymb}
\usepackage[english]{babel}
\usepackage[left=2.5cm, top=3cm, bottom=3cm, right=2.5cm]{geometry}
\usepackage[all, cmtip]{xy}

\def\mymathfont{\mathbf}

\newcommand{\R}{\mymathfont{R}}
\newcommand{\C}{\mymathfont{C}}
\newcommand{\Z}{\mymathfont{Z}}

\newcommand{\SO}{\mathop{\mathrm{SO}}\nolimits}

\newcommand{\CP}{\mathop{\mymathfont{C}\mymathfont{P}}\nolimits}

\newcommand{\M}{\mathop{\mathcal M}\nolimits}

\newcommand{\spinc}{\mathop{\mathrm{spin}^{\mathrm{c}}}\nolimits}

\newcommand{\pt}{\mathop{\mathrm{pt}}\nolimits}

\newcommand{\lb}{\mathop{\mathcal O}\nolimits_{\CP^1}}

\usepackage{euscript,color}

\newtheoremstyle{thm}{6pt plus 1pt minus 1pt}{6pt plus 1pt minus 1pt}{\slshape}{}{\scshape}{.}{5pt plus 1pt minus 1pt}{}
\newtheoremstyle{def}{6pt plus 1pt minus 1pt}{6pt plus 1pt minus 1pt}{}{}{\scshape}{.}{5pt plus 1pt minus 1pt}{}
\newtheoremstyle{rmk}{6pt plus 1pt minus 1pt}{6pt plus 1pt minus 1pt}{}{}{\scshape}{.}{5pt plus 1pt minus 1pt}{}
\newtheoremstyle{claim}{6pt plus 1pt minus 1pt}{6pt plus 1pt minus 1pt}{}{}{\slshape}{.}{5pt plus 1pt minus 1pt}{}

\theoremstyle{thm}
\newtheorem{newstatement}{newstatement}

\newtheorem{theorem}[newstatement]{Theorem}

\theoremstyle{def}
\newtheorem{definition}[newstatement]{Definition}

\theoremstyle{rmk}
\newtheorem{remark}[newstatement]{Remark}

\theoremstyle{claim}

\makeatletter
\renewcommand{\section}{\@startsection%
{section}
{1}
{0mm}
{1.5\bigskipamount}
{0.5\bigskipamount}
{\centering\normalsize\sc}}

\renewcommand{\subsection}{\@startsection%
{subsection}
{2}
{0mm}
{0.5\bigskipamount}
{0.5\bigskipamount}
{\normalsize\sc}}

\renewcommand{\paragraph}{\@startsection%
{paragraph}
{4}
{0mm}
{\bigskipamount}
{0pt}
{\normalsize\sl}}

\expandafter\let\expandafter\oldproof\csname\string\proof\endcsname
\let\oldendproof\endproof
\renewenvironment{proof}[1][\proofname]{%
  \oldproof[\slshape #1]%
}{\oldendproof}
\def\provedboxcontents#1{$\square$}

\let\emph\textsl

\title
{Non-K\"{a}hler complex structures on $\R^4$ II}

\author[Di Scala, Kasuya and Zuddas]{Antonio J Di Scala, Naohiko Kasuya and Daniele Zuddas}
\address{Antonio J Di Scala, Dipartimento di Scienze Matematiche `G.L. Lagrange', Politecnico di Torino, Corso Duca degli Abruzzi 24, 10129, Torino, Italy.}
\email{antonio.discala@polito.it}
\address{Naohiko Kasuya, School of Social Informatics, Aoyama Gakuin University, 5-10-1\break Fuchinobe, Chuo-ku, Sagamihara, Kanagawa 252-5258, Japan.}
\email{nkasuya@si.aoyama.ac.jp}
\address{Daniele Zuddas, Korea Institute for Advanced Study, 85 Hoegiro, Dongdaemun-gu, Seoul 02455, Republic of Korea.}
\email{zuddas@kias.re.kr}

\date{\today}


\begin{document}

\begin{abstract}
We follow our study of non-K\"{a}hler complex structures on $\R^4$  that we defined in \cite{DKZ15}.
We prove that these complex surfaces do not admit any smooth complex compactification. Moreover, we give an explicit description of their meromorphic functions. We also prove that the Picard groups of these complex surfaces are uncountable, and give an explicit description of the canonical bundle. Finally, we show that any connected non-compact oriented 4-manifold admits complex structures without K\"{a}hler metrics.
\end{abstract}

\maketitle

\section{Introduction}

In \cite{DKZ15} we constructed a two-parameters family of pairwise not biholomorphic complex structures $J(\rho _1,\rho _2)$ on $\R^4$, where $1 < \rho_2 < \rho_1^{-1}$. It is remarkable that these complex surfaces cannot be tamed by any symplectic structure. In particular, they have no compatible K\"{a}hler metrics. In addition, there is a surjective map $f \colon \R^4 \to \CP^1$, which is holomorphic with respect to any $J(\rho _1,\rho _2)$.
We denote by $E(\rho _1,\rho _2)$ the complex surface $(\R^4, J(\rho _1,\rho _2))$.
By following \cite{DKZ15}, we give the following definition.

\begin{definition}\label{CE-type/def}
A complex manifold $M$ is said to be of \emph{Calabi-Eckmann type} if there is a holomorphic immersion $k \colon X \to M$, with $X$ a compact complex manifold of positive dimension, such that $k$ is null-homotopic as a continuous map. Such a compact complex immersed submanifold is said to be  {\sl homotopically trivial}.
\end{definition}

Notice that if $M$ is of Calabi-Eckmann type, then it cannot be tamed by a symplectic form, and hence $M$ cannot be endowed with a K\"ahler metric.

The motivation for this definition is that Calabi and Eckmann in \cite{CE53} constructed such complex structures on $\R^{2n}$, for $n \geq 3$, arising as open subsets of certain compact complex $n$-manifolds (which are diffeomorphic to the product of two odd-dimensional spheres).

On the other hand, the manifolds $E(\rho_1, \rho_2)$ are the first examples of Calabi-Eckmann type complex surfaces \cite{DKZ15}. However, as the following main theorem states, our surfaces cannot be realised as analytic subsets of any compact complex non-singular surface.

\begin{theorem}\label{compact}
$E(\rho _1,\rho _2)$ cannot be holomorphically embedded in any compact complex surface.
\end{theorem}

For a complex manifold $X$, let $\M(X)$ denote the field of meromorphic functions of $X$. Moreover, we denote by $\mathrm{Pic}(X)$ the Picard group of $X$, and by $\omega_X$ the canonical line bundle of $X$.

Let $\lb(k)$ be the holomorphic line bundle on $\CP^1$ with first Chern number $k \in \Z$, and consider the induced holomorphic line bundle $L_k=f^{\ast }(\lb(k))$ on $E(\rho _1,\rho _2)$. By abusing notation, we denote by $f^*$ also the pullback homomorphisms determined by $f$ between the groups that we consider in the next theorem. This shall not be misleading, being the exact meaning of the symbol $f^*$ clear from the context.

\begin{theorem}\label{meromorphic}
The following properties hold:
\begin{enumerate}
\item $f^{\ast } \colon   \M(\CP^1)\to  \M(E(\rho _1, \rho _2))$ is an isomorphism.

\item $f^{\ast } \colon  \mathrm{Pic}(\CP^1) \cong \Z \to \mathrm{Pic}(E(\rho _1, \rho _2))$ is injective and maps $\lb(-2)$ to $\omega_{E(\rho _1,\rho _2)}$, namely $\omega_{E(\rho _1,\rho _2)} = L_{-2}$.

\item $\mathrm{Pic}(E(\rho _1, \rho _2))$ is uncountable.
\end{enumerate}
\end{theorem}

\begin{theorem}\label{R^{2n}}
The total space of a holomorphic vector bundle of rank $n$ over $E(\rho_1, \rho_2)$ determines a Calabi-Eckmann type complex structure on $\R ^{2n+4}$. Moreover, the total spaces of the Whitney sums $L_{k_1} \oplus L_{k_2}\oplus \cdots \oplus L_{k_n}$, with $k_1\leq k_2\leq \cdots \leq k_n$, are pairwise not biholomorphic to each other for all choices of the parameters $1 < \rho_2 < \rho_1^{-1}$ and $(k_1, k_2, \cdots ,k_n) \in \Z^n$.
\end{theorem}

\begin{remark}
The complex structures on $\R^{2n}$ that arise as in the statement of Theorem~\ref{R^{2n}} are not biholomorphic to those constructed by Calabi and Eckmann in \cite{CE53}. Indeed, by our construction, $\R^{2n}$ admits an immersed rational curve with one node, while the only compact curves in the Calabi and Eckmann examples are smooth elliptic curves.
\end{remark}

Finally, we give a general result for Calabi-Eckmann type complex surfaces.

\begin{theorem}\label{noncpt/thm}
Any non-compact connected oriented smooth 4-manifold $M$ admits a complex structure of Calabi-Eckmann type, which is compatible with the given orientation.
\end{theorem}


The paper is organized as follows. In Section \ref{DKZ} we recall the construction of the complex manifolds $E(\rho _1,\rho _2)$ as it is given in \cite{DKZ15}, and we also fix the notations. In Section \ref{Classes}, we recall some classification results for compact complex surfaces that are needed in the proof of Theorem~\ref{compact}. In Section \ref{proofs} we prove Theorems~\ref{compact} and \ref{meromorphic}. Finally, in Section \ref{2n} we prove Theorems \ref{R^{2n}} and \ref{noncpt/thm}.

\section{Non-K\"{a}hler complex structures on $\R^4$}\label{DKZ}

Consider the complex annulus
$\Delta(r_0, r_1) = \{z \in \C \mid r_0 < |z| < r_1 \}$ and the disk $\Delta(r) = \{z \in \C \mid \allowbreak |z| < r \}$. Fix positive numbers $\rho_0$, $\rho_1$ and $\rho_2$ such that $1 < \rho_2 < \rho_1^{-1} < \rho_0^{-1}$.

Let $f_1 \colon W_1 \to \Delta (\rho _1)$ be a relatively minimal elliptic holomorphic Lefschetz fibration with one singular fiber. In other words, $W_1$ is a fibered tubular neighborhood of a singular fiber of type $I_1$ in an elliptic complex surface by Kodaira results \cite{Kod60, Kod63b}.

Following Kodaira \cite{Kod63}, we can represent $f_1$ as follows.
First, we consider the quotient $W_1' = (\C^{\ast }\times \Delta (0, \rho _1))\slash \Z$, where the action is given by $n\cdot (z, w)=(zw^n, w)$.
The projection on the second factor $W'_1 \to \Delta (0, \rho _1)$ is a holomorphic elliptic fibration that is equivalent to $f_1$ restricted over $\Delta (0, \rho _1)$. Therefore, we can consider $W_1'$ as an open subset of $W_1$ (precisely, as the complement of the singular fiber).

Let $W_2 = \Delta (1, \rho _2) \times \Delta (\rho _0^{-1})$ be endowed with the product complex structure, and let $f_2 \colon W_2 \to \Delta (\rho _0^{-1})$ be the canonical projection.
Consider also the open subset $V_2 = \Delta (1, \rho _2) \times \Delta (\rho _1^{-1}, \rho _0^{-1})\subset W_2$.

Now, consider the multi-valued holomorphic function $\varphi  \colon \Delta (\rho _0, \rho _1)\to \C ^{\ast }$ defined by
\begin{eqnarray*}
\varphi(w)=\exp \left(\frac{1}{4\pi i}(\log w)^2 - \frac{1}{2} \log w \right).
\end{eqnarray*}

As it can be easily checked, the image of the map $\Phi \colon \Delta(1, \rho_2) \times \Delta(\rho_0, \rho_1) \to \C^* \times \Delta(0, \rho_1)$ defined by $\Phi(z,w) = (z \varphi(w), w)$, is invariant under the above $\Z$-action. Therefore, the composition of $\Phi$ with the projection map $\pi \colon \C^* \times \Delta(0, \rho_1) \to W_1' \subset W_1$ determined by the $\Z$-action, is a single-valued holomorphic embedding, and let $V_1 \subset W_1$ be the image of $\pi \circ \Phi$. Notice that $\pi \circ \Phi$ gives an identification of $V_1$ with the product $\Delta(1, \rho_2) \times \Delta(\rho_0, \rho_1)$.

We are now ready to glue $W_1$ and $W_2$ together by identifying $V_1$ and $V_2$ via the biholomorphism $j \colon V_1 = \Delta(1, \rho_2) \times \Delta(\rho_0, \rho_1) \to V_2$ defined by $j(z, w) = (z, w^{-1})$. We obtain the complex manifold $E(\rho _1,\rho _2) = W_1 \cup_j W_2$. The result does not depend on the parameter $\rho_0$ up to biholomorphisms, since $\rho_0$ determines only the size of the gluing region.

Moreover, we can define a surjective holomorphic map $f \colon E(\rho _1,\rho _2)\to \CP^1$
by putting $f(x) = f_1(x)$ if $x \in W_1$ and $f(x) = f_2(x)$ if $x \in W_2$, where $\CP^1$ is regarded as the Riemann sphere resulting from the gluing $\Delta(\rho_1) \cup_h \Delta(\rho_0^{-1})$, with $h \colon \Delta(\rho_0, \rho_1) \to \Delta(\rho_1^{-1}, \rho_0^{-1})$ defined by $h(w) = w^{-1}$.

It can be proved that $E(\rho _1,\rho _2)$ is diffeomorphic to $\R^4$ essentially by means of the Kirby calculus applied to Lefschetz fibrations, see Gompf and Kirby \cite{GS99} and Apostolakis, Piergallini and Zuddas \cite{APZ2013}. In \cite{DKZ15} this has been done by relating the decomposition of $E(\rho_1, \rho_2)$ given in its definition with a decomposition of $\R^4 \subset S^4$ determined by the Matsumoto-Fukaya fibration, which is a genus-one achiral Lefschetz fibration $S^4\to S^2$ \cite{Mat82}. Actually, the map $f$ above coincides with the restriction of this fibration on $\R^4 \subset S^4$.

It follows that the fibers of $W_1$ have null-homotopic inclusion map in the contractible space $E(\rho _1,\rho _2)$, namely they are homotopically trivial. Then, the complex manifold $E(\rho _1,\allowbreak\rho _2) \cong \R^4$ is of Calabi-Eckmann type by Definition \ref{CE-type/def}.

Moreover, the only compact holomorphic curves in $E(\rho_1, \rho_2)$ are the compact fibers of $f$, see \cite[Proposition 5]{DKZ15}.

\section{The classification of compact complex surfaces}\label{Classes}

According to Kodaira \cite[Theorem 55]{Kod68} (see also \cite{Kod60, Kod68a}), compact complex surfaces are classified into the following seven classes, up to blow ups and blow downs.
\begin{enumerate}
\item[(I)] $\CP^2$ or $\CP^1$-bundles over a compact complex curve;
\item[(II)] $K3$ surfaces;
\item[(III)] complex tori;
\item[(IV)] K\"{a}hler elliptic surfaces;
\item[(V)] algebraic surfaces of general type;
\item[(VI)] elliptic surfaces whose first Betti number $b_1$ is odd and greater than $1$;
\item[(VII)] surfaces with $b_1=1$.
\end{enumerate}

Surfaces of class I to V are K\"{a}hler, while surfaces of classes VI and VII are non-K\"{a}hler, because of the following theorem.

\begin{theorem}[Miyaoka \cite{Mi74} and Siu \cite{Si83}]
A compact complex surface is K\"{a}hler if and only if the first Betti number is even.
\end{theorem}

We recall also the following theorem.

\begin{theorem}[Chow and Kodaira \cite{CK52} and Kodaira \cite{Kod60}]
A compact complex surface is algebraic if the algebraic dimension is two, and it is elliptic if the algebraic dimension is one.
\end{theorem}

Thus, the algebraic dimension of a non-K\"{a}hler surface is $0$ or $1$.
Hence it is $1$ for class VI.
On the other hand, class VII is divided into the two cases where the algebraic dimension is $0$ or $1$.
If it is $0$, then $X$ has at most finitely many compact holomorphic curves by Kodaira \cite{Kod60}.
Otherwise, $X$ is an elliptic surface.
Such elliptic surfaces are further classified according as if they are minimal or not minimal.
If $X$ is a minimal elliptic surface of class VII, then $X$ is obtained from the product $\CP^1\times E$ by a finite sequence of logarithmic transformations, where $E$ is a smooth elliptic curve \cite{Kod66}.
If $X$ is not minimal, then it is obtained from a minimal one by blow ups.

\section{The proofs of Theorems \ref{compact} and \ref{meromorphic}}\label{proofs}

\begin{proof}[Proof of Theorem \ref{compact}]
The proof is by contradiction.
So, suppose that $E(\rho_1,\rho_2)$ is embedded as a domain in a compact complex surface $X$, and let $i \colon  E(\rho_1,\rho_2)\to X$ be a holomorphic embedding.

Recall that $E(\rho_1,\rho_2)$ is not K\"ahler. Hence, $X$ cannot be K\"ahler,
and so it is of class VI or VII. Moreover, the algebraic dimension is $0$ or $1$.
If it is $0$, then $X$ contains at most finitely many compact holomorphic curves, while $E(\rho _1,\rho _2)$ does contain infinitely many, which is impossible.
Hence, the algebraic dimension must be $1$. Therefore, $X$ has an elliptic fibration $g \colon X \to \Sigma$ over a compact complex curve.

Notice that any compact fiber $F$ of $f \colon E(\rho_1, \rho_2) \to \CP^1$ is homotopically trivial in $E(\rho_1, \rho_2)$, because $E(\rho_1, \rho_2)$ is contractible. Then $(g \circ i)_{|F} \colon F \to \Sigma$ is holomorphic and null-homotopic as a continuous map, and so it is of zero degree, implying that $g \circ i$ is constant on $F$. Indeed, a holomorphic non-constant map between compact Riemann surfaces is open (hence surjective), and it has positive degree because it is orientation-preserving at the regular points.

Therefore, by compactness, $i(F)$ is a fiber of $g$. In other words, $i$ is fiber-preserving in the elliptic part of $E(\rho_1, \rho_2)$.

It follows that $g \circ i$ is constant on any fiber of $f$, even in the non-compact ones, because the vertical derivative of $g \circ i$, that is the tangent map restricted to the kernel of $Tf$, is null in the elliptic part of $E(\rho_1, \rho_2)$, which is an open subset of $E(\rho_1, \rho_2)$, and so, by analyticity, it must be null everywhere.
Therefore, $i \colon E(\rho _1,\rho _2)\to X$ is a fiberwise embedding.
This also implies that $\Sigma $ must be $\CP^1$, and we can regard $f \colon E(\rho _1,\rho _2)\to \CP^1$ as a restriction of the fibration $g \colon  X\to \CP^1$ by identifying $E(\rho _1,\rho _2)$ with its image $i(E(\rho _1,\rho _2)) \subset X$.

Now, we first consider the case when $X$ is of class VI.
In this case, $X$ is an elliptic surface over $\CP^1$ such that $b_1(X) \geq 3$.
Since $\pi_1(T^2)=\Z ^2$ and $\CP^1$ is simply connected, the first Betti number $b_1(X)$ must be smaller than $3$.
This is a contradiction.

Next, we consider $X$ of class VII.
Suppose that $X$ is a minimal elliptic surface.
Such surfaces are classified as in Section \ref{Classes}. There is an elliptic fibration $\pi \colon  X\to \CP^1$ having only finitely many multiple fibers that are smooth elliptic curves, and no singular fibers.

On the other hand, the fibration $f \colon E(\rho _1,\rho _2)\to \CP^1$ has a singular fiber which is an immersed sphere with one node, and by the above argument we have $f = \pi \circ i$. Thus, $\pi$ has a singular fiber, which is a contradiction.

Finally, we consider the case where $X$ is a non-minimal elliptic surface.
Let $C\subset X$ be an exceptional curve of the first kind.
Since the rational curve $C$ cannot be a fiber of $\pi  \colon X\to \CP^1$, it must be a branched multi-section of $\pi $.
Then, $C$ has a positive intersection number with an elliptic fiber of $f \colon E(\rho _1, \rho _2)\to \CP^1$.
This is a contradiction, since $E(\rho _1,\rho _2)$ is contractible.
\end{proof}

\begin{remark}\label{enoki/rmk}
Ichiro Enoki, in a private communication, told us that there is no compact complex surface of Calabi-Eckmann type. It is proven as follows.  If a compact complex surface contains a compact holomorphic curve representing a homologically trivial $2$-cycle, then it must be a surface of Class VI, a Hopf surface, a parabolic Inoue surface, or an Enoki surface \cite{En80}. In each case, the holomorphic curve contains a loop representing a generator of the fundamental group of the surface.
This fact leads to an alternative proof of Theorem \ref{compact}.
\end{remark}

\begin{proof}[Proof of Theorem \ref{meromorphic}]
We begin with statement (1). Let $h$ be a meromorphic function on $E(\rho _1,\rho _2)$ and let $(h) = (h)_0 - (h)_{\infty}$ be the associated divisor, see Griffiths and Harris \cite[pages 36, 131]{GH78}. Then, $(h)_0$ and $(h)_{\infty}$ are codimension-$1$ analytic subsets and the indeterminacy set $N = (h)_0 \cap (h)_{\infty}$ is an analytic subset of codimension greater than $1$, implying that it is a finite set.

The restriction of the meromorphic function $h$ to $E(\rho _1,\rho _2) - N$ is a holomorphic map $h_| \colon  E(\rho _1,\rho _2) - N\to \CP^1$.
All but finitely many compact fibers of $f$ are contained in $E(\rho _1,\rho _2) - N$,
and they are also homotopically trivial (cf. the end of Section \ref{DKZ}) in the complement of $N$.

By the same argument in the proof of Theorem \ref{compact}, we see that $h$ is constant on the compact fibers of $f$, and then it is constant on all fibers of $f$.
Therefore, $h$ is the pullback of a meromorphic function $s$ on $\CP^1$, implying that $f^* \colon \M(\CP^1)\to  \M(E(\rho _1, \rho _2))$ is surjective. On the other hand, $f^*$ is clearly injective, and this concludes the proof of (1).

\smallskip
Now, we prove the statement (2).
As in the Introduction, let $\lb(k)$ denote the holomorphic line bundle over $\CP^1$ with first Chern number $k \in \Z$. It is well known that $\lb(k)$ is holomorphically trivial if and only if $k = 0$.

Consider the decomposition $E(\rho_1, \rho_2) = W_1 \cup_j W_2$ as in Section \ref{DKZ}.
So, $W_1$ and $W_2$ are the preimages by $f$ of two open disks $\Delta_1 = \Delta(\rho_1)$ and $\Delta_2 = \Delta(\rho_0^{-1})$ whose union is $\CP^1$.
We have to show that if $L_k = f^*(\lb(k))$ is trivial, then $k = 0$.

Since $\lb(k)$ is holomorphically trivial over each of the disks $\Delta_1$ and $\Delta_2$, there are nowhere vanishing holomorphic sections $\sigma_i$ of $\lb(k)$ over $\Delta_i$, for $i=1, 2$. Then, the pullback section $f^*(\sigma_i)$ is a trivialization of $L_k$ over $W_i$, $i=1, 2$.

If $L_k$ is holomorphically trivial, there is a nowhere vanishing global holomorphic section $\tau$. Hence, there are holomorphic functions $\tau_i \in \mathcal{O}^*(W_i)$, $i=1,2$, such that \[ \tau_{|W_i} = \tau_i f^*(\sigma_i) \,\]

Now, we show that $\tau _1$ and $\tau _2$ are the pullbacks of holomorphic functions on $\Delta _1$ and $\Delta _2$.
On $W_1$, the fibers of $f$ are all compact holomorphic curves. Hence, $\tau _1$ is constant along the fibers of $f$.
Thus, $\tau _1$ is the pullback $f^{\ast }(u_1)$ of some holomorphic function $u_1$ on $\Delta _1$.
Then, on the intersection $V = W_1\cap W_2$, we have
\[f^{\ast }(u _1\sigma _1)=\tau _2f^{\ast }(\sigma _2). \,\]
Hence, $\tau _2$ is fiberwise constant on $V$. By analyticity, it is fiberwise constant over $W_2$.
Hence, there is a holomorphic function $u_2$ on $\Delta _2$ such that $\tau _2=f^{\ast }(u_2)$.
Then, $u_1\sigma _1$ and $u_2\sigma _2$ define a nowhere vanishing holomorphic global section of $\lb(k)$.
Hence $\lb(k)$ is trivial, and so $k = 0$.

Next, we prove that $\omega_{E(\rho _1,\rho _2)} = L_{-2}$.

Set $Y = E(\rho _1,\rho _2)$. The canonical line bundle $\omega_Y$ is by definition the determinant
of the contangent bundle $T^* Y$, that is $\omega_Y = \Lambda^2 (T^* Y)$.

Let $(z,w)$ be the coordinates of $\C^{\ast }\times \Delta (\rho _1)$. An easy computation shows that the holomorphic $2$-form
\[ \sigma = \frac{\mathrm{d} z \wedge \mathrm{d} w }{z}\]
passes to quotient $\C^{\ast }\times \Delta (0, \rho _1)\slash \Z$, where, as above, the action is given by $n\cdot (z,w)=(zw^n,w)$.

Since $\sigma $ is defined on $\C^{\ast }\times \Delta (\rho _1)$, it induces a regular holomorphic $2$-form on $$\big(\C^{\ast }\times \Delta (0, \rho _1)\slash \Z \big)\cup \C^{\ast }= W_1 - \left\{q \right\},$$ where $q$ is the nodal singularity of the singular elliptic fiber. By Hartog's theorem, it can be extended over the point $q$.
Hence, $\sigma $ determines a regular holomorphic $2$-form on the whole elliptic fibration $W_1$.

Let $(u,t)$ be the coordinates of  $\Delta (1, \rho _2) \times \Delta (\rho _0, \rho _1)$. By the above identification of $V$ with $\Delta (1, \rho _2) \times \Delta (\rho _0, \rho _1)$ we see that:
\[ \sigma = \frac{\mathrm{d} u \wedge \mathrm{d} t }{u}\]
on $\Delta (1, \rho _2) \times \Delta (\rho _0, \rho _1)$.
By means of the biholomorphism $j$ we get that
\[ \sigma = -\frac{\mathrm{d} z}{ z} \wedge \frac{\mathrm{d} s}{s^2} \]
where here $(z,s)$ are coordinates of $\Delta (1, \rho _2) \times \Delta (\rho _0^{-1})$.

Thus, $\sigma $ gives rise to a meromorphic section of the canonical bundle of $Y$.
The polar set of $\sigma $ is $2F_2$, where $F_2$ is an annulus fiber of the map $f$. Hence we obtain that
$$\omega_Y = f^*(\lb(-2)).$$

\smallskip
Finally, we prove statement (3).
For a complex manifold $Y$, it is well known that the exponential function gives rise to the following short exact sequence of sheaves:
\[ 0 \to \mathbb{Z} \to \mathcal{O}_Y \to {\mathcal{O}_Y^*} \to 0 \, \, .\]

For $Y = E(\rho _1,\rho _2)$, being diffeomorphic to $\R^4$, we obtain the following isomorphism from the associated long exact sequence of cohomology groups
$$\mathrm{Pic}(Y) = \mathrm{H}^1(Y ,\mathcal{O}_Y^*) \cong  \mathrm{H}^1(Y,\mathcal{O}_Y).$$

By definition, $\mathrm{H}^1(Y,\mathcal{O}_Y)$ is a complex vector space.
Hence, $\mathrm{Pic}(E(\rho _1,\rho _2))$ is isomorphic to the additive group of a complex vector space. We have already proved that $\mathrm{Pic}(E(\rho _1,\rho _2))$ is not trivial. Therefore, it is uncountable.
\end{proof}

\section{The proofs of Theorems \ref{R^{2n}} and \ref{noncpt/thm}}\label{2n}

\begin{proof}[Proof of Theorem \ref{R^{2n}}] The first sentence of the statement is straightforward, because the Calabi-Eckamnn type surface $E(\rho_1, \rho_2)$ is holomorphically embedded in the total space as the zero section of the bundle. We prove the second sentence.

Let $Q^{n+2}(\rho_1, \rho_2, k) \cong \R^{2n+4}$ be the total space of the holomorphic vector bundle $L_{k_1}\oplus \cdots \oplus L_{k_n}$, and let $\pi_k$ be its projection map, that is $\pi_k \colon Q^{n+2}(\rho_1, \rho_2, k) \to E(\rho_1, \rho_2)$, with $k = (k_1,\dots, k_n)$.

Put $\xi_k = \lb(k_1) \oplus\cdots \oplus \lb(k_n)$. We have $\pi_k = f^*(\xi_k)$.
Then, $$Q^{n+2}(\rho_1, \rho_2, k) = \{(x,y) \in E(\rho_1, \rho_2) \times E(\xi_k) \mid f(x) = \xi_k(y)\},$$ where we denote by $E(\xi_k)$ the total space of $\xi_k$. The map $\pi_k$ is the projection on the first factor, and we denote by $\tilde f \colon Q^{n+2}(\rho_1, \rho_2, k) \to E(\xi_k)$ the projection on the second factor, as it is shown in the following commutative diagram.
$$\xymatrix{
Q^{n+2}(\rho_1, \rho_2, k) \ar[d]_{\pi_k} \ar[r]^-{\tilde f}&   E(\xi_k) \ar[d]^{\xi_k}\\
E(\rho_1, \rho_2) \ar[r]_-f&  \CP^1}$$

Next, we classify compact connected holomorphic curves on $Q^{n+2}(\rho_1,\allowbreak \rho_2,\allowbreak k)$.
Let $S$ be a compact Riemann surface and let $\iota  \colon  S\to Q^{n+2}(\rho_1, \rho_2, k)$ be a holomorphic immersion.
The holomorphic map $\xi_k \circ \tilde f \circ \iota  \colon  S\to \CP^1$ is null-homotopic, because $Q^{n+2}(\rho_1, \rho_2, k)$ is contractible, and hence it is constant (cf. the proof of Theorem \ref{compact} again).

Therefore, $(\tilde f \circ \iota)(S)$ is contained in a fiber of $\xi_k$, which is biholomorphic to $\C^n$. This implies that $\tilde f \circ \iota$ is constant.
Hence, $\pi_{k} \circ \iota \colon S \to E(\rho_1, \rho_2)$ must be a holomorphic immersion.
By the classification of compact holomorphic curves on $E(\rho_1, \rho_2)$, $(\pi_{k} \circ \iota) (S)$ is a compact fiber of $f$.

It follows that compact connected holomorphic curves in $Q^{n+2}(\rho_1, \rho_2, k)$ are of the form $F_p \times \{y\}$ where $F_p = f^{-1}(p)$ is a compact fiber of $f$, and $\xi_k(y) = p$, for some $p \in \CP^1$.

\smallskip

Now, suppose that there exists a biholomorphism $\Phi : Q^{n+2}(\rho_1, \rho_2, k)\to Q^{n+2}(\rho_1', \rho_2', k')$, where $k' = (k_1', \dots, k_n') \in \Z^n$ is another $n$-tuple with not decreasing components, and $1 < \rho_2' < (\rho_1')^{-1}$.
We are going to show that $(\rho _1, \rho _2, k)=(\rho _1', \rho _2', k')$.

By keeping the above notation, we consider the map $f' \colon E(\rho_1', \rho_2') \to \CP^1$ of the construction in Section \ref{DKZ}, the bundle $\pi_{k'}$, and the projection map $\tilde f' \colon Q^{n+2}(\rho_1', \rho_2', k') \to E(\xi_{k'})$ such that $\xi_{k'} \circ \tilde f' = f' \circ \pi_{k'}$.

Let $F_p \times \{y\}$ be a compact connected holomorphic curve in $Q^{n+2}(\rho_1, \rho_2, k)$. Then, $\Phi(F_p \times \{y\}) = F'_{p'} \times \{y'\}$ for certain $p' \in \CP^1$ and $y' \in E(\xi_{k'})$ such that $\xi_{k'}(y') = p'$, with $F'_{p'} = (f')^{-1}(p')$.

By the construction of Section \ref{DKZ}, there is a biholomorphism $F_p \cong F'_{p'}$ if and only if $p = p'$ (by considering the moduli of elliptic fibers).

By analyticity, the equality $\Phi(F_p \times \{y\}) = F'_{p} \times \{y'\}$ must hold for all $p \in \CP^1$ and for all $y \in \xi_k^{-1}(p)$, where $y' \in \xi_{k'}^{-1}(p)$ is uniquely determined by $p$ and $y$. Therefore, there are two biholomorphisms $\varphi \colon E(\rho_1, \rho_2) \to E(\rho_1', \rho_2')$ and $\psi \colon E(\xi_k) \to E(\xi_{k'})$, such that $\Phi(x, y) = (\varphi(x), \psi(y))$ for all $(x, y) \in Q^{n+2}(\rho_1, \rho_2, k)$.
An application of the main theorem of \cite{DKZ15} yields $(\rho_1, \rho_2) = (\rho_1', \rho_2')$.

Note that $\psi $ is a fiberwise biholomorphism (that is, $\xi_{k'} \circ \psi = \xi_k$), but is not necessarily a linear bundle isomorphism. So, we take the fiber derivative of $\psi$ evaluated along the zero section of $\xi_k$.
By identifying the fibers of $\xi_k$ and $\xi_{k'}$ with the corresponding tangent spaces at a suitable point,
we obtain a linear isomorphism $\xi_k \cong \xi_{k'}$ of holomorphic vector bundles. Then, the Birkhoff-Grothendieck theorem \cite{Gr57} tells us that $k = k'$.
\end{proof}

\begin{proof}[Proof of Theorem \ref{noncpt/thm}]
Teichner and Vogt proved in an unpublished paper that any oriented 4-manifold admits a $\spinc$-structure (see Gompf and Stipsicz \cite[Remark 5.7.5]{GS99} for a sketch of their proof). By Gompf \cite{Go97}, a $\spinc$-structure is a homotopy class of complex structures over the 2-skeleton that are extendable over the 3-skeleton. Since $M$ is non-compact, it has the homotopy type of a 3-complex. It follows that $M$ admits an almost complex structure that is compatible with the given orientation.

There is a nowhere zero vector field on $M$, because $M$ is non-compact. Thus, the tangent bundle $TM$, regarded as a rank-two complex vector bundle, splits as the Whitney sum of a complex line bundle $\xi$ and a complex trivial line bundle $\varepsilon^1$, that is $TM = \xi \oplus \varepsilon^1$.

Let $t \colon M \to \CP^\infty$ be a classifying map for $\xi$, so that $\xi$ is the pullback of the tautological line bundle over $\CP^\infty$. Up to homotopy we can assume that $t$ takes values in the 3-skeleton of $\CP^\infty$, which is $\CP^1$.

Then, $TM$ is isomorphic, as a real vector bundle, to a pullback of an oriented non-trivial rank-four real vector bundle over $S^2$. Since $\pi_1(\SO(4))=\Z_2$, there is only one such bundle up to isomorphisms, which is equivalent to the restriction to $\CP^1 \subset \overline{\CP}^2 - \{\pt\}$ of $T(\overline{\CP}^2 - \{\pt\})$.

Hence, there is a monomorphism of real vector bundles $G \colon T M \to T (\overline{\CP}^2 - \{\pt\})$. By Phillips theorem \cite{P67}, there is an orientation-preserving immersion $g \colon M \to \overline{\CP}^2 - \{\pt\}$.

Notice that $\overline{\CP}^2 - \{\pt\}$ admits a Calabi-Eckmann type complex structure, obtained by taking the blow up of $E(\rho_1, \rho_2)$ at a point (cf. \cite[Corollary 6]{DKZ15}). Let us denote by $P(\rho_1, \rho_2) \cong \overline{\CP}^2 - \{\pt\}$ such blow up. It is still true that any holomorphic torus in $P(\rho_1, \rho_2)$ away from the blow up point is contained in a 4-ball.

Now, take a 4-ball $D \subset M$ where the restriction of $g$ is an embedding. Up to isotopy we can assume that $g(D)$ contains a holomorphic torus of $P(\rho_1, \rho_2)$. Then, the complex structure on $M$ induced by $g$ is of Calabi-Eckmann type.
\end{proof}

\begin{remark}
The proof of Theorem \ref{noncpt/thm} can be slightly modified to show that any non-compact connected oriented 4-manifold $M$ can be immersed into $\CP^2$, implying the known fact that $M$ admits a
K\"ahler complex structure.
With Theorem~\ref{noncpt/thm}, we have shown that any non-compact connected oriented 4-manifold admits both
K\"ahler and non-K\"ahler complex structures.
\end{remark}

\section*{Acknowledgements}

Antonio J Di Scala and Daniele Zuddas are members of GNSAGA of INdAM.

The authors would like to thank Prof. Ichiro Enoki for Remark \ref{enoki/rmk}.

\end{document}